\documentclass[10pt, leqno]{article}

\usepackage[utf8x]{inputenc}
\usepackage{amsfonts}
\usepackage{amsmath}
\usepackage{amssymb}
\usepackage{amsthm}
\usepackage{mathrsfs}
\usepackage{latexsym}
\usepackage{graphicx}
\usepackage{bm}
\usepackage{inputenc}
\usepackage{enumerate}
\usepackage{enumitem}
\usepackage{hyperref}
\usepackage{graphicx}

\swapnumbers

\newtheorem{Lemma}{Lemma}[section]
\newtheorem{Corollary}[Lemma]{Corollary}
\newtheorem{Theorem}[Lemma]{Theorem}

\theoremstyle{definition}

\newtheorem{Definition}[Lemma]{Definition}

\theoremstyle{remark}
\newtheorem{Remark}[Lemma]{Remark}

\newtheoremstyle{proof*}
{3pt}
{3pt}
{\rmfamily}
{}
{\bfseries}
{.}
{.5em}
{\thmnote{#3}}
\theoremstyle{proof*}
\newtheorem*{proof*}{}

\DeclareMathOperator{\restrict}{\llcorner}

\DeclareMathOperator{\Tan}{Tan}

\DeclareMathOperator{\Nor}{Nor}
\DeclareMathOperator{\Hom}{Hom}     
\DeclareMathOperator{\Der}{D}       
\DeclareMathOperator{\pt}{pt}       
\DeclareMathOperator{\ap}{ap}

\DeclareMathOperator{\trace}{trace}

\newcommand{\Real}[1]{ \mathbf{R}^{#1}}
\newcommand{\Haus}[1]{ \mathscr{H}^{#1} }
\newcommand{\Leb}[1]{ \mathscr{L}^{#1} }
\newcommand{\rect}[1]{(\mathscr{H}^{#1},#1)}

\newcommand{\Lp}[1]{\mathbf{L}^{#1}}

\title{ABP inequalities for singular submanifolds of bounded mean curvature}
\author{Mario Santilli}

\begin{document}
	\maketitle
	
	\begin{abstract}
		Employing a notion of curvature for arbitrary closed sets we prove an ABP-type estimate for a class of singular submanifolds of arbitrary codimension and bounded mean curvature recently introduced by B.\ White. A weak-Harnack-type estimate is then derived using the ABP estimate. These results generalize analogous results by O.\ Savin for viscosity solutions of the minimal surface equation.   
	\end{abstract}

	
	\section{Introduction}
	
We prove ABP (Alexandroff-Bakelmann-Pucci) inequalities for a class \textit{singular submanifolds of arbitrary codimension} in the Euclidean space which includes the class of varifolds without boundary and with bounded mean curvature. 

More specifically, we consider the following class.

\begin{Definition}\label{definition of (m,h)}
	Suppose $ 1 \leq m \leq n $ are integers, $ \Omega $ is an open subset of $ \Real{n+1} $, $ \Gamma $ is relatively closed in $ \Omega $ and $ h \geq 0 $. We say that $ \Gamma $ is a $ (m,h) $ subset of $ \Omega $ provided it has the following property: if $ x \in \Gamma $ and $ f $ is a $ \mathscr{C}^{2} $  function in a neighbourhood of $ x $ such that $ f|\Gamma $ has a local maximum at $ x $ and $ \nabla f(x) \neq 0 $, then
	\begin{equation*}
	\trace_{m}\Der^{2}f(x) \leq h |\nabla f(x)|,
	\end{equation*}
where $ \trace_{m} \Der^{2}f(x) $ is the sum of the lowest $ m $ eigenvalues of $ \Der^{2}f(x) $.
\end{Definition}

This class has been recently introduced\footnote{This definition is equivalent to \cite[2.1]{MR3466806}, by \cite[8.1]{MR3466806}.} in \cite[2.1]{MR3466806} in the analysis of the area blow-up of sequences of smooth submanifolds (or varifolds) with a uniform bound on the mean curvature. In fact, as proved in \cite[2.6]{MR3466806}, if $ M_{i} $ is a sequence of $ m $ dimensional varifolds in an open subset $ \Omega $ of $ \Real{n+1} $ with mean curvatures uniformly bounded in $ \Lp{\infty} $ by a non-negative number $ h $ and such that
\begin{equation*}
\limsup_{i \to \infty}|\partial M_{i}|(U) < \infty  \quad \textrm{whenever $ U \subset \subset \Omega $},
\end{equation*}
where $ |\partial M_{i}| $ is the measure associated to the generalized boundary of $ M_{i} $, then the smallest closed subset $ Z $ of $ \Omega $ such that the areas of $ M_{i} $ are uniformly bounded as $ i \to \infty $ on compact subsets of $ \Omega \sim Z $ is an $ (m,h) $ subset of $ \Omega $. As a simple consequence of this general fact, it follows that the support of every $ m $ dimensional varifold in $ \Omega $ with no boundary and mean curvature bounded in $ \Lp{\infty} $ by $ h $ is an $(m,h)$ subset of $ \Omega $, see \cite[2.8]{MR3466806}.

Special cases of \ref{definition of (m,h)}, under the name of viscosity minimal sets, have been studied in \cite{2017arXiv170507948S} and \cite{2017arXiv170801549S}.

We describe now the results of the present paper. 

Firstly we introduce some notation\footnote{For basic notation, see the end of the introduction.} that we use through the rest of this work. Suppose $ \mathscr{C}_{1}(0)=\mathbf{U}^{n}(0,1) \times \mathbf{R} \subseteq \mathbf{R}^{n+1} $, $ \Gamma \subseteq \mathscr{C}_{1}(0) $ is relatively closed and $ a > 0 $. For every $ x \in \mathbf{R}^{n} $ we define $ P_{a,x} $ to be the paraboloid of opening $ a $ and center $ x $ touching $ \overline{\Gamma} $ from above. For every closed set $ C \subseteq \mathbf{B}^{n}(0,1) $ we define the touching set
\begin{equation*}
A_{a}(\Gamma; C)
\end{equation*}
as the set of $ (z,\eta) $ such that $z \in \Sigma(P_{a,x})  \cap \overline{\Gamma}$ for some $ x \in C $  and $ \eta \in \Nor(\Sigma(P_{a,x}), z) $ with $ |\eta |=1 $ and $ \eta_{n+1} > 0 $. Moreover, we let
\begin{equation*}
A'_{a}(\Gamma; C) = \{ z' : \textrm{$(z, \eta) \in A_{a}(\Gamma; C)$ for some $ \eta \in \mathbf{S}^{n} $}\}.
\end{equation*}
When $ C = \mathbf{B}^{n}(0,1) $ we let $ A_{a}(\Gamma; C) = A_{a}(\Gamma) $ and $ A'_{a}(\Gamma;C) = A'_{a}(\Gamma) $.

For submanifolds of arbitrary codimension we obtain the following ABP inequality:

\begin{Theorem}[ABP estimate, arbitrary codimension]\label{ABP arbitrary cod}
	Suppose $ 1 \leq m \leq n $ are integers, $ h \geq 0 $, $ a > 0 $ and $ \Gamma $ is an $ (m,h) $ subset of $ \mathscr{C}_{1}(0) $ such that $ \Haus{m}(\Gamma) < \infty $. For every closed set $ C $ in $ \mathbf{B}^{n}(0,1) $, if $ \varnothing \neq A_{a}(\Gamma; C) \subseteq \Gamma \times \mathbf{S}^{n} $, then
	\begin{equation*}
		\Haus{n}(C) \leq \gamma \Big(1+a+\frac{1}{a}\Big)^{n-m}\Big(1 + a + \frac{h}{a}\Big)^{m} \int_{\Gamma^{(m)}}\Haus{n-m}\{ \eta : (z,\eta) \in A_{a}(\Gamma; C) \} d\Haus{m}z, 
	\end{equation*}
	 where $ \Gamma^{(m)} $ is the set of $ z \in \Gamma $ where $ \Gamma $ can be touched from $ n+1-m $ linearly independent directions by an open ball of $ \Real{n+1} $ and $ \gamma = 4^{n-m}(2m + 4h)^{m} $.
\end{Theorem}

In case of hypersurfaces, the previous result can be refined as follows:

 \begin{Corollary}[ABP estimate, codimension 1]\label{ABP}
	Suppose $ m = n $ in \ref{ABP arbitrary cod}. For every closed set $ C $ in $ \mathbf{B}^{n}(0,1) $, if $ \varnothing \neq A_{a}(\Gamma; C) \subseteq \Gamma \times \mathbf{S}^{n} $, then
	\begin{equation}\label{ABP:2}
	\Haus{n}(C) \leq \gamma (1+a+ha^{-1})^{n}(1+4a^{2})^{1/2}\Haus{n}(A'_{a}(\Gamma; C)),
	\end{equation}
	where $ \gamma = (2n + 4h)^{n} $.
\end{Corollary}

As a special case of \ref{ABP} we obtain the following ABP estimate proved in \cite[2.1]{MR2334822} for viscosity subsolutions of the minimal surface equation\footnote{The results in \cite{MR2334822} are actually proved for a class of elliptic operators $ F $ that are uniformly elliptic in a neighbourhood of the origin and $ F(0,p,z,x) =0 $ whenever $ p $, $ z $ and $ x $ are close to $ 0 $.} .

\begin{Corollary}[Savin]\label{Savin ABP}
	Let $ u $ be a continuous viscosity subsolution in $ \mathbf{U}^{n}(0,1) $ of the minimal surface equation, let $ \Gamma \subseteq \Real{n+1} $ be the graph of $ u $ and let $ 0 < a \leq 1 $. 
	
	Then there exists a small universal constant $ \gamma $ (depending only on $ n $) such that for every closed set $ C \subseteq \mathbf{B}^{n}(0,1) $ such that $ A'_{a}(\Gamma; C) \subseteq \mathbf{U}^{n}(0,1) $,
	\begin{equation*}
	\Haus{n}(A'_{a}(\Gamma;C)) \geq \gamma \Haus{n}(C).
	\end{equation*}
\end{Corollary}
Similar estimates have been proved in several other works since then; see, for example, \cite[3.5]{MR3500837} \cite[6.5]{MR3695374}, \cite[3.5]{MR3265174} and, in a Riemannian setting, \cite[Theorem 1.2]{MR2989991}. 

It is well known, starting with the prioneering work of Krilov and Safanov, that the ABP estimate plays a fundamental role in the derivation of Harnack-type estimates. In this regard \ref{ABP} can be employed to extend to $ (n,h) $ sets in $ \mathbf{R}^{n+1} $ the Measure-Estimate in \cite[6.1]{MR2757359} (the same estimate is implicitly contained in the first part of the proof of the Harnack inequality in \cite[1.1]{MR2334822}). This type of estimate is also known as weak Harnack inequality, see \cite[6.2]{MR3695374}. Having the ABP estimate at our disposal the proof goes along the line of \cite[2.2]{MR2334822}.

\begin{Theorem}[Measure-Estimate]\label{Extrinsic weak Harnack inequality}
	For every $ n \geq 1 $ there exist $ \alpha > 1 $ and $ 0 < \beta < 1 $ such that the following statement holds. For every $ \mu > 0 $ there exists an integer $ k \geq 1 $ so that if $ \Gamma $ is an $(n,h)$ subset of $ \mathcal{C}_{1}(0) $ with $ \Haus{n}(\Gamma) < \infty $, $ 0 \leq h < \alpha^{-k-1} $,
	\begin{equation*}
	\Gamma \subseteq \{ x : x_{n+1} \leq 0  \} \quad \textrm{and} \quad \Gamma \cap \{ x : |x'|\leq 1/4, \; -\alpha^{-k-1}/48 \leq x_{n+1} \leq 0 \} \neq \varnothing,
	\end{equation*}
	then
	\begin{equation*}
	\Leb{n}(\mathbf{B}^{n}(0, 1/3) \sim A'_{\alpha^{-1}}(\Gamma)) \leq \mu \Leb{n}(\mathbf{B}^{n}(0, 1/3)).
	\end{equation*}
\end{Theorem}

We finally remark that the proof of \ref{ABP arbitrary cod} and \ref{ABP}, though based on the same general idea of \cite[2.1]{MR2334822}, use a substantially different method to treat the arbitrary-codimension setting of \ref{ABP arbitrary cod}. This method is based on the theory of curvature for arbitrary closed sets recently developed in \cite{2017arXiv170801549S} (see also \cite{MR534512}, \cite{MR2031455} for earlier contributions in this direction). The proof of \cite[2.1]{MR2334822} goes in the following way: firstly the problem is reduced to semi-concave functions by inf-convolution approximation; since such functions are differentiable on the contact set, it is possible to introduce an auxiliary function (\emph{vertex map}), mapping each contact point into the vertex of the corresponding touching paraboloid; the vertex map is Lipschitz and its derivative can be suitably estimated using the assumptions on the elliptic operator $ F $; henceforth, an application of the classical area formula for Lipschitz maps gives the conclusion. In the present paper we also use a vertex map, but now its domain is a subset of the normal bundle of the $ (m,h) $ set $ \Gamma $, namely $ A_{a}(\Gamma; C) $. Then, in order to complete the proof following the idea of \cite[2.1]{MR2334822}, we need a suitable area formula for functions whose domain is a subset of the normal bundle, see \ref{general facts}\eqref{general facts: Coarea}, and suitable estimates for the differential of the vertex map, that can be deduced from the upper bound on the trace of the second fundamental form in \ref{trace and viscosity minimal sets}.

\paragraph*{Notation.}The distance function from a closed set $ \Gamma $ in $ \Real{n+1} $ is $ \bm{\delta}_{\Gamma} $ and the nearest point projection onto $ \Gamma $ is $ \bm{\xi}_{\Gamma} $. If $T$ is a linear subspace of $ \Real{n+1} $ then $ T_{\natural} $ is the orthogonal projection onto $ T $.

If $ z \in \Real{n} \times \Real{} \cong \Real{n+1} $ we define $ z' \in \Real{n} $ so that \mbox{$(z', z_{n+1}) = z $.}
	If $ \rho > 0 $ and $ x \in \Real{n} $ we define 
	\begin{equation*}
	\mathcal{C}_{\rho}(x) = \Real{n+1} \cap \{ z :  |z'-x| < \rho  \},
	\end{equation*}
	\begin{equation*}
	\mathbf{B}^{n}(x,\rho) = \Real{n} \cap \{y : | y-x| \leq \rho \}, \quad \mathbf{U}^{n}(x,\rho) = \Real{n} \cap \{y : | y-x| < \rho \}.
	\end{equation*}
	
	For every function $ f : D \rightarrow \Real{} $, where $ D \subseteq \Real{n} $, we let
	\begin{equation*}
	\Sigma(f) = \{ (x,f(x)) : x \in D  \} \quad \textrm{and} \quad \Sigma^{+}(f) = \{ (x,y) : y > f(x) \}.
	\end{equation*}
	
	If $ a > 0 $ and $ x \in \Real{n} $ then a function $ P : \Real{n} \rightarrow \Real{} $ is a \emph{paraboloid of center $ x $ and opening $ a $} if there exists $ t \in \Real{} $ such that
	\begin{equation*}
	P(y) = (a/2)|y-x|^{2} + t \quad  \textrm{whenever $ y \in \Real{n} $}.
	\end{equation*}
	
		If $ \Gamma \subseteq \Real{n+1} $ is a closed set and $ f : \Real{n} \rightarrow \Real{} $ is continuous, we say that \emph{$ f $ touches $ \Gamma $ from above} if and only if
	\begin{equation*}
	\Sigma(f) \cap \Gamma \neq \varnothing \quad \textrm{and} \quad \Sigma^{+}(f)\cap \Gamma = \varnothing.
	\end{equation*}

	\section{Curvature of closed sets} \label{Curvature of minimal varieties}
	
	Suppose $ n \geq 1 $ is an integer and $ \Gamma \subseteq \Real{n+1} $ is closed. 

For every $ r > 0 $ we define
		\begin{equation*}
		N_{r}(\Gamma)= (\Gamma \times \mathbf{S}^{n}) \cap \{ (z,\eta): \bm{\delta}_{\Gamma}(z+r\eta) = r \};
		\end{equation*}
the \emph{generalized normal bundle of $ \Gamma $} is then given by
		\begin{equation*}
		N(\Gamma) = \bigcup_{r > 0}N_{r}(\Gamma).
		\end{equation*}
		
		If $ 0 \leq m \leq n+1 $ the \emph{$ m $-th stratum of $ \Gamma $} is 
		\begin{equation*}
	\Gamma^{(m)} = \Gamma \cap	\{ z : \dim \bm{\xi}_{\Gamma}^{-1}\{z\} = n+1-m \}.
		\end{equation*}
		It is not difficult to see that $ \Gamma^{(m)} $ is the set of $ z \in \Gamma $ where $ \Gamma $ can be touched from $ n+1-m $ linearly independent directions by an open ball of $ \Real{n+1} $; see \cite[5.1-5.3]{2017arXiv170801549S}. 
		
We prove in \cite[3.10(1)]{2017arXiv170801549S} that $ N_{r}(\Gamma) $ is a locally $ \rect{n} $ rectifiable closed subset of $ \Real{n+1} \times \mathbf{S}^{n} $ and we deduce that $ N(\Gamma) $ is a countably $ \rect{n} $ rectifiable Borel subset of $ \Real{n+1} \times \mathbf{S}^{n} $ since $ N_{r}(\Gamma) \subseteq N_{s}(\Gamma) $ if $ s < r $. We remark that $ N(\Gamma) $ coincides with the normal bundle introduced in \cite[\S 2.1]{MR2031455} and it is well known that if $ \Gamma $ is a set of positive reach then $ N(\Gamma) $ coincides with the classical unit normal bundle, as defined in \cite[3.1.21]{MR0257325}, and it has locally finite $ \Haus{n} $ measure (see \cite[4.4]{2017arXiv170801549S}).

In \cite[4.5-4.12]{2017arXiv170801549S} we prove that if $ \Gamma $ is an arbitrary closed subset of $ \Real{n+1} $ then for $ \Haus{n} $ a.e.\ $(z,\eta) \in N(\Gamma) $ there exist a linear subspace $ T_{\Gamma}(z,\eta) $ of $ \Real{n+1} $ with $ \dim  T_{\Gamma}(z,\eta) \geq 0 $ and a bilinear form $ Q_{\Gamma}(z,\eta) :T_{\Gamma}(z,\eta) \times T_{\Gamma}(z,\eta) \rightarrow \Real{} $ that coincides with the second fundamental form introduced in \cite[4.5]{MR1021369} when $ \Gamma $ is a set of positive reach. Henceforth we refer to this bilinear form as \emph{the second fundamental form of $ \Gamma $.} The principal curvatures $ - \infty <	\kappa_{1}(z,\eta) \leq \ldots \leq \kappa_{n-1}(z,\eta) \leq \infty $ of $ \Gamma $ at $ (z,\eta) $ can be defined as
\begin{equation*}
\kappa_{m+1}(z,\eta) = \infty, \quad \textrm{$ \kappa_{1}(z,\eta), \ldots , \kappa_{m}(z,\eta) $ are the eigenvalues of $ Q_{\Gamma}(z,\eta) $}
\end{equation*}
where $ m = \dim T_{\Gamma}(z,\eta) $. They coincide $ \Haus{n} $ almost everywhere in $ N(\Gamma) $ with the principal curvatures introduced in \cite[p.\ 244]{MR2031455} by different methods, see \cite[4.16]{2017arXiv170801549S}. We remark that the functions mapping $ \Haus{n} $ almost all $ (z,\eta) \in N(\Gamma) $ into $ T_{\Gamma}(z,\eta)_{\natural} \in \Hom(\Real{n+1}, \Real{n+1})  $ and $ Q_{\Gamma}(z, \eta) \circ \bigodot_{2}T_{\Gamma}(z,\eta)_{\natural} \in \Hom(\bigodot_{2}\Real{n+1}) $ are $ \Haus{n} $ measurable.

For reader's convenience in the next theorem we summarize two facts from the general theory developed in \cite{2017arXiv170801549S} that we employ in the proof of the ABP estimate.

\begin{Theorem}\label{general facts}
	Let $ \Gamma $ be a closed subset of $ \Real{n+1} $. Then the following statements hold.
	\begin{enumerate}
		\item\label{general facts:basis} For every $ r > 0 $ and for $ \Haus{n} $ a.e.\ $ (z,\eta) \in N_{r}(\Gamma) $ the approximate tangent cone $ \Tan^{n}(\Haus{n}\restrict N_{r}(\Gamma), (z,\eta)) $ is an $ n $-dimensional plane and there exist $ u_{1}, \ldots , u_{n} \in \Real{n+1} $ such that $ \{ u_{1}, \ldots , u_{n}, \eta \} $ is an orthonormal basis of $\Real{n+1}$ and the orthonormal vectors
		\begin{equation*}
		\bigg(\frac{1}{(1+\kappa_{i}(z,\eta)^{2})^{1/2}} u_{i}, \frac{\kappa_{i}(z,\eta)}{(1+\kappa_{i}(z,\eta)^{2})^{1/2}} u_{i} \bigg) 
		\end{equation*}
 form a basis of $ \Tan^{n}(\Haus{n}\restrict N_{r}(\Gamma), (z,\eta)) $ (if $ \kappa_{i}(z, \eta) = \infty $ then the corresponding vector equals $ (0, u_{i}) $).
 \item\label{general facts: Coarea} If $f$ is a $ (\Haus{n}\restrict N(\Gamma)) $ integrable $ \overline{\Real{}} $ valued function, then
 \begin{flalign*}
 & \int_{N(\Gamma)|\Gamma^{(m)}} f(a,u)\, \prod_{i=1}^{m}\frac{1}{(1+\kappa_{i}(a,u)^{2})^{1/2}} \, d\Haus{n}(a,u) \\
 & \quad = \int_{\Gamma^{(m)}}\int_{\{z\} \times N(\Gamma,z)}f \, d\Haus{n-m}\, d\Haus{m}z.
 \end{flalign*}
\end{enumerate}
\end{Theorem}

\begin{proof}
	See \cite[4.14, 5.6]{2017arXiv170801549S}.
\end{proof}

In order to combine this measure-theoretic theory with techniques developed in PDE's theory, we have to analize the behaviour of the principal curvatures introduced above in order to produce useful estimates. 

At this point it is worth to recall that if $ \kappa_{1} $ is the principal curvature of the graph $ \Gamma $ of the primitive of the (ternary) Cantor function, which is a convex $ \mathscr{C}^{1} $ function from $ \Real{} $ to $ \Real{} $, then, letting $ L = \bigcap_{r > 0}N_{r}(\Gamma) $, there exists $ M \subseteq L $ such that $ \Haus{1}(M)> 0 $ and 
		\begin{equation*} 
\kappa_{1}(z,\eta) = \infty \quad \textrm{for every $(z,\eta) \in M $.}
		\end{equation*}
		The reader may consult \cite[5.11]{2017arXiv170801549S} for details on this example.
		
		In order to prove the main result of this section, firstly we need a generalization of the barrier principle in \cite[7.1]{MR3466806}.

\begin{Lemma}\label{weak maximum principle}
	Suppose $ 1 \leq m \leq n $ are integers, $ f : \Real{n} \rightarrow \Real{} $ is function pointwise differentiable of order $ 2 $ at $ 0 $ such that $ f(0) =0 $ and $ \Der f(0) =0 $, $ h \geq 0 $, $ \Omega $ is an open subset of $ \Real{n+1} $, $ \Gamma $ is an $ (m,h) $ subset of $ \Omega $ such that $ 0 \in \Gamma $ and 
	\begin{equation*}
		\Gamma \cap V \subseteq \{ z : z_{n+1}\leq f(z')  \}
	\end{equation*}
	for some open neighbourhood $ V $ of $ 0 $.  
	Then, denoting by $ \chi_{1} \geq \ldots \geq \chi_{n} $ the eigenvalues of $ \Der^{2}f(0) $, it follows that
	\begin{equation*}
 \chi_{1} + \ldots + \chi_{m} \geq -h.
	\end{equation*}
\end{Lemma}

\begin{proof}
	Let $ \epsilon > 0 $ and $ \psi(x) = \frac{1}{2}\Der^{2}f(0)(x,x) + \epsilon |x|^{2} $ for $ x \in \mathbf{R}^{n} $. There exists $ r > 0 $ such that $ f(x)\leq \psi(x) $ for every $ x \in \mathbf{U}^{n}(0,r) $. By \cite[7.1]{MR3466806}, if $ \kappa_{1} \leq \ldots \leq \kappa_{n} $ are the principal curvatures of $ M = \{(x, \psi(x)) : x \in \mathbf{R}^{n}   \} $ at $ (0,0) $ with respect to the unit normal that points into $ \{ z : z_{n+1} \leq \psi(z')  \} $, then
	\begin{equation*}
		\kappa_{1} + \ldots + \kappa_{m} \leq h.
	\end{equation*}
	Since a standard and straightforward computation shows that $ \kappa_{i} = -\chi_{i} -\epsilon $ for $ i = 1, \ldots , n $, we obtain the conclusion letting $ \epsilon \to 0 $.
\end{proof}

The main result of this section provides the crucial estimate for the trace of the second fundamental form of an $ (m,h) $ set. As the example of the primitive of the Cantor function clearly shows, this estimate cannot be expected to hold without a specific geometric hypothesis (as the weak maximum principle for $ (m,h) $ sets). The proof of the following theorem is adapted from \cite[7.5]{2017arXiv170801549S}, where certain stationary submanifolds (viscosity minimal sets) are treated.
	 
	\begin{Theorem}\label{trace and viscosity minimal sets}
		If $ \Omega $ is an open subset in $ \Real{n+1} $, $\Gamma$ is an $ (m,h) $ subset of $ \Omega $ that is a countable union of sets with finite $ \Haus{m} $ measure and if $ \kappa_{1} \leq \ldots \leq \kappa_{n} $ are the principal curvatures \mbox{of $ \overline{\Gamma} $,} then $ \kappa_{m+1}(z, \eta) = \infty $ and
		\begin{equation*}
\textstyle	\sum_{i=1}^{m}\kappa_{i}(z,\eta) =	\trace Q_{\overline{\Gamma}}(z, \eta) \leq h
		\end{equation*}
		for $ \Haus{n} $ a.e.\ $ (z, \eta) \in N(\overline{\Gamma}) \cap \{ (z, \nu) : z \in \Gamma \} $. In particular, if $ \Gamma $ is the support of an $ m $-dimensional integral varifold in $ \Omega $ without boundary and with bounded variational mean curvature $ H $, then
		\begin{equation*}
	\textstyle	\sum_{i=1}^{m}\kappa_{i}(z,\eta) =	\trace Q_{\overline{\Gamma}}(z, \eta) = -H(z) \bullet \eta
	\end{equation*}
	for $ \Haus{n} $ a.e.\ $ (z, \eta) \in N(\overline{\Gamma}) \cap \{ (z, \nu) : z \in \Gamma \} $.
	\end{Theorem}

\begin{proof}
Suppose $ A = \overline{\Gamma} $, $ \lambda > (2m-1) + 1 $ and the functions $ \bm{\nu}_{A} $ and $ \bm{\psi}_{A} $ are defined as in \cite[3.1]{2017arXiv170801549S}. Given such $ \lambda $ we define for $ r > 0 $ the set $ N_{r} $ as in \cite[3.10]{2017arXiv170801549S}. We recall from the proof of \cite[4.14]{2017arXiv170801549S} that for $ \Leb{1} $ a.e.\ $ r > 0 $ and for $ \Haus{n} $ a.e.\ $ x \in N_{r} $, if $ \chi_{1} \leq \ldots \leq \chi_{n} $ are the eigenvalues of $ \ap \Der \bm{\nu}_{A}(x)| \Tan^{n}(\Haus{n}\restrict N_{r}, x) $ then 
\begin{equation}\label{trace and viscosity minimal sets:1}
	\kappa_{i}(\bm{\psi}_{A}(x)) = \chi_{i}(1-r\chi_{i})^{-1} \quad \textrm{for $ i = 1, \ldots , \dim T_{A}(\bm{\psi}_{A}(x)) $.}
\end{equation}
	
	 We select $ 0 < r < (2h)^{-1} $ so that \cite[3.10(3),(4)]{2017arXiv170801549S} hold for $ \Haus{n} $ a.e.\ $ x \in N_{r} $; then we fix $ x \in N_{r} \cap \bm{\xi}_{A}^{-1}(\Gamma) $ satisfying the conclusions of \cite[3.10(2),(3),(4)]{2017arXiv170801549S}. We assume $ \bm{\xi}_{A}(x) =0 $ and let $ T = \{ v : v \bullet \bm{\nu}_{A} (x) =0 \} $. It follows that there exists a Lipschitzian function $ f : T \rightarrow T^{\perp} $ such that $ f $ is pointwise differentiable of order $ 2 $ at $ 0 $, $ \Der f(0) =0 $, 
	\begin{equation*}
	\Der^{2}f(0)(u,v) \bullet \bm{\nu}_{A}(x) = -\ap\Der \bm{\nu}_{A}(x)(u) \bullet v \quad \textrm{whenever $ u , v \in T $,}
	\end{equation*}
	\begin{equation*}
	W \cap \bm{\delta}_{A}^{-1}\{r\} = W \cap \{ \chi + f(\chi) : \chi \in T \} \quad \textrm{for some neighbourhood $ W $ of $ x $.}
	\end{equation*}
	Choose $ 0 < s < r/2 $ such that $ \mathbf{U}(x,s) \subseteq W $ and let $ g(\zeta) = f(\zeta) - x $ for $ \zeta \in T $, 
		\begin{equation*}
	U= T_{\natural}\big(\mathbf{U}(x,s) \cap \{ \chi + f(\chi) : \chi \in T \}\big), \quad V = \{ y-x: y \in T_{\natural}^{-1}(U) \cap \mathbf{U}(x,s)  \}.
	\end{equation*}
	Then $ V $ is an open neighbourhood of $ 0 $ and 
	\begin{equation*}
	V \cap A \subseteq \{ z : z \bullet \bm{\nu}_{A}(z) \leq g(T_{\natural}(z)) \bullet \bm{\nu}_{A}(z) \}.
	\end{equation*}
	This inclusion can be proved as follows: if there was $y \in \mathbf{U}(x,s) \cap T_{\natural}^{-1}[U]$ such that $y-x \in A $ and $ y \bullet \bm{\nu}_{A}(x) > f(T_{\natural}(y))\bullet \bm{\nu}_{A}(x) $ then, noting that $ \bm{\nu}_{A}(x) = r^{-1}x $,
	\begin{equation*}
	T_{\natural}(y) + f(T_{\natural}(y)) \in \mathbf{U}(x,s) \cap \bm{\delta}_{A}^{-1}\{r\}, \quad |T_{\natural}(y) + f(T_{\natural}(y))-y| < r,
	\end{equation*}
	we would conclude
	\begin{equation*}
	|T_{\natural}(y) + f(T_{\natural}(y))-(y-x)| = r - (y-f(T_{\natural}(y)))\bullet \bm{\nu}_{A}(x)< r
	\end{equation*}
	which is a contradiction. Let $ \chi_{1} \leq \ldots \leq \chi_{n} $ be the eigenvalues of $\ap \Der \bm{\nu}_{A}(x)|T$. Then $ 1 - \chi_{1}r, \ldots , 1 - \chi_{n}r $ are the eigenvalues of $ \ap \Der \bm{\xi}_{A}(x)|T$ by \cite[3.5]{2017arXiv170801549S} and $ -\chi_{1}, \ldots , - \chi_{n} $ are the eigenvalues of $ \pt\Der^{2} g(0) \bullet \bm{\nu}_{A}(x) $. Therefore we apply a rotated version of \ref{weak maximum principle} to conclude that 
	\begin{equation}\label{trace and viscosity minimal sets:2}
 \chi_{1} + \ldots + \chi_{m} \leq h.
	\end{equation}
	Since $ \chi_{j} \geq -(\lambda -1)^{-1}r^{-1}$ whenever $ j = 1 , \ldots , n $ by \cite[3.10(2)]{2017arXiv170801549S}, we get 
	\begin{equation*}
	\chi_{j} - (m-1)(\lambda-1)^{-1}r^{-1} \leq \sum_{i=1}^{m} \chi_{i} \leq h, 
	\end{equation*}
	whence we conclude $ \chi_{j} < r^{-1} $ whenever $ j = 1 , \ldots , m $. Therefore,
	\begin{equation*}
		\| \textstyle \bigwedge_{m}\big((\Haus{n}\restrict N_{r},n)\ap \Der \bm{\xi}_{A}(x)\big) \| > 0.
	\end{equation*}
	We remind that the assertions proved so far hold for $\Haus{n}$ a.e.\ $ x \in N_{r} \cap \bm{\xi}_{A}^{-1}(\Gamma) $ and for $ \Leb{1} $ a.e.\ $ 0 < r < (2h)^{-1} $.
	
We prove now the following \emph{Lusin $(N)$ condition}: if $ S \subseteq \Gamma $ and $ \Haus{m}(S \cap A^{(m)}) =0 $ then $ \Haus{n}(N(A) \cap \{ (x,\nu) : x \in S\}  ) =0 $. Using \cite[3.10(1), 4.3, 5.3]{2017arXiv170801549S} we get that
\begin{equation*}
A \cap	\{ x : \Haus{n-m}(\bm{\xi}_{A}^{-1}\{x\} \cap N_{r} ) >0  \} \subseteq \bigcup_{i=0}^{m}A^{(i)}
\end{equation*}
for every $ r > 0 $ and, noting that $ \Haus{m}(A^{(i)}) =0 $ for $ i = 0, \ldots , m-1 $ by \cite[5.2]{2017arXiv170801549S}, we deduce
\begin{equation*}
\Haus{m}\big( S \cap \{ x : \Haus{n-m}(\bm{\xi}_{A}^{-1}\{x\} \cap N_{r} ) >0  \} \big) =0.
\end{equation*}
Noting \cite[3.10(1)(2)]{2017arXiv170801549S}, we apply \cite[7.4]{2017arXiv170801549S} with $ W $ and $ f $ replaced by $ N_{r} \cap \bm{\xi}_{A}^{-1}(\Gamma) $ and $ \bm{\xi}_{A} $ to get
\begin{equation*}
\Haus{n}(\bm{\xi}_{A}^{-1}(S) \cap N_{r}   ) =0 \quad \textrm{for $ \Leb{1} $ a.e.\ $ 0 < r < (2h)^{-1} $,}
\end{equation*}
whence the desired Lusin condition follows from \cite[4.3]{2017arXiv170801549S}.

By \cite[5.2, 5.9]{2017arXiv170801549S} it follows that $ \dim T_{A}(z,\eta) = m $ and $ \kappa_{m+1}(z,\eta) = \infty $ for $ \Haus{n} $ a.e.\ $ (z, \eta) \in N(A) \cap \{ (z, \nu) : z \in \Gamma \} $. Moreover we combine \eqref{trace and viscosity minimal sets:1} and \eqref{trace and viscosity minimal sets:2} to infer that
\begin{equation}\label{trace and viscosity minimal sets:3}
\sum_{i=1}^{m}\frac{\kappa_{i}(\bm{\psi}_{A}(x))}{1 + r \kappa_{i}(\bm{\psi}_{A}(x))} \leq h 
\end{equation}
for $ \Haus{n} $ a.e.\ $ x \in N_{r} \cap \bm{\xi}_{A}^{-1}(\Gamma) $ and for $ \Leb{1} $ a.e.\ $ 0 < r < (2h)^{-1} $. We choose a sequence $ r_{i} > 0 $ converging to $ 0 $ such that if $ M_{i} $ is the set of $ x \in N_{r_{i}} \cap \bm{\xi}_{A}^{-1}(\Gamma) $ such that \eqref{trace and viscosity minimal sets:3} is satisfied, then $ \Haus{n}(N_{r} \cap \bm{\xi}_{A}^{-1}(\Gamma) \sim M_{i}) =0 $ for every $ i \geq 1 $. Then we readily infer that 
\begin{equation*}
	\trace Q_{A}(z,\eta) \leq h
\end{equation*}
for every $ (z,\eta) \in \bigcap_{i=1}^{\infty}\bigcup_{j=i}^{\infty}\bm{\psi}_{A}(M_{i}) $ and, noting that $ \bm{\psi}_{A}(N_{r}) \subseteq \bm{\psi}_{A}(N_{s}) $ if $ s < r $, we additionally get that
\begin{equation*}
\textstyle	\Haus{n}\Big( N(A) \cap \{ (z, \nu) : z \in \Gamma \} \sim \bigcap_{i=1}^{\infty}\bigcup_{j=i}^{\infty}\bm{\psi}_{A}(M_{i}) \Big) =0.
\end{equation*}

If $ \Gamma $ is the support of an $ m $ dimensional integral varifold in $ \Omega $ without boundary and with mean curvature bounded in $ \Lp{\infty} $ by $ h $, then, noting that $ \Gamma $ is an $ (m,h) $ subset of $ \Omega $ by \cite[2.8]{MR3466806}, we may achieve the conclusion using the Lusin $(N)$ condition above in combination with \cite[5.2, 5.9]{2017arXiv170801549S} and \cite[4.2]{MR2472179}.
\end{proof}

\begin{Remark}
We may conjecture that \ref{trace and viscosity minimal sets} could also be proved for classes of \emph{singular submanifolds of possibly unbounded mean curvature}. For example, if $ \Gamma $ is the support of an $ m $ dimensional integral varifold in $ \Real{n+1} $ without boundary and with mean curvature $ H $ in $ \Lp{m} $, is it true that $ \kappa_{m+1}(z,\eta) = \infty $ and
\begin{equation*}
\textstyle \sum_{i=1}^{m}\kappa_{i}(z,\eta) = - H(z) \bullet \eta 
\end{equation*}
for $ \Haus{n} $ a.e.\ $ (z,\eta) \in N(\Gamma) $? More generally, what can be concluded in case of rectifiable varifolds with a uniform lower bound on the density? To treat these cases the method of \ref{trace and viscosity minimal sets} is not powerful enough and new refined tools seem to be necessary.
\end{Remark}

\section{Proof of \ref{ABP arbitrary cod} and \ref{ABP}}

We notice that $ A_{a}(\Gamma; C) $, $ A'_{a}(\Gamma; C) $ and 
\begin{equation*}
	\{ z :  \textrm{$(z, \eta) \in A_{a}(\Gamma; C)$ for some $ \eta \in \mathbf{S}^{n} $}  \}
\end{equation*}
are closed subset of $ \overline{\Gamma} \times \mathbf{S}^{n} $, $ \mathbf{B}^{n}(0,1) $ and $ \overline{\Gamma} $ respectively.


	\begin{proof*}[Proof of \ref{ABP arbitrary cod}]
	Clearly, for every $ (w,\eta) \in A_{a}(\Gamma; C) $ there exists a unique $ x \in C $ such that $ w \in \Sigma(P_{a,x}) $ and $ \eta \in \Nor(\Sigma(P_{a,x}),w) $; in fact
		\begin{equation*}
		x = w' + a^{-1}\eta_{n+1}^{-1}\eta'.
		\end{equation*}
		Henceforth, we define the map $ \Psi : A_{a}(\Gamma; C) \rightarrow C $ so that $ \Psi(w,\eta)=  w' + \frac{1}{a\eta_{n+1}}\eta' $ and we notice that $ \Psi[A_{a}(\Gamma; C)]= C $. 
		
	It is not difficult to check that\footnote{In fact, a straightforward computation shows that \emph{if $ n \geq 1 $ is an integer, $ P $ is a paraboloid of center $ 0 $ and opening $ a > 0 $, $ x \in \Real{n} $ and $ \eta \in \Nor(\Sigma(P), (x, P(x)) ) $ with $ |\eta |= 1 $ and $ \eta_{n+1} > 0 $, then
		\begin{equation*}
		\sup\{s : \bm{\delta}_{\Sigma(P)}((x,P(x)) + s \eta   ) = s   \} = (a\eta_{n+1})^{-1}.
		\end{equation*}}}
		\begin{equation*}
		\bm{\delta}_{\overline{\Gamma}}\big(w + (a\eta_{n+1})^{-1}\eta\big) = (a\eta_{n+1})^{-1}
		\end{equation*}
		whenever $ (w, \eta) \in A_{a}(\Gamma; C) $. Henceforth $ A_{a}(\Gamma; C) \subseteq N_{a^{-1}}(\overline{\Gamma}) $ and, if  $ \kappa_{1}, \ldots , \kappa_{n} $ are the principal curvatures of $ \overline{\Gamma} $, it follows from \cite[4.12]{2017arXiv170801549S} that
		\begin{equation}\label{ABP arbitrary cod:1}
		\kappa_{i}(w, \eta) \geq - a\eta_{n+1}
		\end{equation}
		for $ \Haus{n} $ a.e.\ $ (w, \eta) \in A_{a}(\Gamma; C) $ and $ i = 1, \ldots , n $. Employing \ref{trace and viscosity minimal sets} we deduce that
		\begin{equation}\label{ABP arbitrary cod:2}
		 \kappa_{i}(w, \eta) \leq (m-1) a \eta_{n+1} + h, \quad \kappa_{m+1}(w,\eta) = \infty,
		\end{equation}
	for $ \Haus{n} $ a.e.\ $ (w,\eta) \in A_{a}(\Gamma; C) $.

We compute now the jacobian of $ \Psi $ using the orthonormal basis provided by \ref{general facts}\eqref{general facts:basis}. For $ \Haus{n} $ a.e.\ $ (w,\eta) \in A_{a}(\Gamma; C) $, if $ u_{1}, \ldots , u_{n} $ are the orthonormal vectors provided by \ref{general facts}\eqref{general facts:basis}, noting that $ \eta_{n+1}^{-1} \leq (1+4a^{2})^{1/2}\leq 1+ 2a $ and using \eqref{ABP arbitrary cod:1}-\eqref{ABP arbitrary cod:2}, we can easily estimate
\begin{equation*}
| \Der \Psi(w,\eta)(0,u_{i})| \leq \frac{1}{a}\Big( \frac{1}{\eta_{n+1}} + \frac{|\eta'|}{\eta_{n+1}^{2}}\Big) \leq \frac{1}{a\eta_{n+1}} + \frac{2}{\eta_{n+1}}\leq 4 + 4a + \frac{1}{a}
\end{equation*}
for $ i = m+1, \ldots , n $, and
\begin{flalign*}
& |\Der \Psi(w,\eta)(u_{i}, \kappa_{i}(w,\eta) u_{i})| \leq  1 +\frac{|\kappa_{i}(w,\eta)|}{a\eta_{n+1}} + \frac{|\kappa_{i}(w,\eta)||\eta'|}{a\eta_{n+1}^{2}}  \\
& \quad \quad \leq  1 +\frac{|\kappa_{i}(w,\eta)|}{a\eta_{n+1}} + 2\frac{|\kappa_{i}(w,\eta)|}{\eta_{n+1}} \leq \gamma \Big(1+a+\frac{h}{a}\Big),
\end{flalign*}
for $ i = 1, \ldots  m $, where $ \gamma = 2m + 4h $. Therefore for $ \Haus{n} $ a.e.\ $ (w,\eta) \in A_{a}(\Gamma; C) $, we get
		\begin{flalign*}
		& \| \textstyle{\bigwedge_{n}} [\Der \Psi(w,\eta)|\Tan^{n}(\Haus{n}\restrict A_{a}(\Gamma; C) , (w,\eta))] \| \\
		& \quad \quad \leq 4^{n-m}\gamma^{m} \Big(1+a+\frac{1}{a}\Big)^{n-m}\Big(1 + a + \frac{h}{a}\Big)^{m} \prod_{i=1}^{m}\frac{1}{(1+\kappa_{i}(w,\eta)^{2})^{1/2}}.
		\end{flalign*}
		Now we employ \cite[3.2.22(3)]{MR0257325}, \ref{general facts}\eqref{general facts: Coarea} and the Lusin $(N)$ condition obtained in the proof of \ref{trace and viscosity minimal sets} to conclude
		\begin{flalign*}
		& \Haus{n}(C) \leq \int_{C}\Haus{0}\{ (w,\eta) : \Psi(w,\eta)= x \} \,d \Haus{n}x \\
		& \quad 	= \int_{A_{a}(\Gamma; C)| \Gamma^{(m)}} \| \textstyle{\bigwedge_{n}} [\Der \Psi(w,\eta)|\Tan^{n}(\Haus{n}\restrict A_{a}(\Gamma; C), (w,\eta))] \| \,d\Haus{n}(w,\eta) \\
		& \quad \leq  \gamma_{1}\int_{A_{a}(\Gamma; C)| \Gamma^{(m)}} \prod_{i=1}^{m}\frac{1}{(1+\kappa_{i}^{2})^{1/2}}  \, d\Haus{n} \\
		& \quad \leq  \gamma_{1} \int_{\Gamma^{(m)}}\Haus{n-m}\{ \eta : (z,\eta) \in A_{a}(\Gamma; C) \} d\Haus{m}z, 
		\end{flalign*}
		where $ \gamma_{1} = 4^{n-m}\gamma^{m} \Big(1+a+\frac{1}{a}\Big)^{n-m}\Big(1 + a + \frac{h}{a}\Big)^{m} $.
		\end{proof*}

			\begin{Lemma}\label{ABP : aux lemma}
				Suppose $ \Gamma \subseteq \mathscr{C}_{1}(0) $ is relatively closed such that $ \Haus{n}(\Gamma) < \infty $, $ a > 0 $, $ C \subseteq \mathbf{B}^{n}(0,1) $ is closed and
				\begin{equation*}
		B=	\{ z :  \textrm{$(z, \eta) \in A_{a}(\Gamma; C)$ for some $ \eta \in \mathbf{S}^{n} $}  \}.
				\end{equation*}
				
				Then 
				\begin{equation*}
					\Haus{n}(B) \leq (1+4a^{2})^{1/2}\Haus{n}(A'_{a}(\Gamma;C)).
				\end{equation*}
			\end{Lemma}
		
		\begin{proof}
	 If $ z \in B $, $ x \in C $ and $ \nu \in \Real{n+1} $ such that $ z \in \Sigma(P_{a,x}) $ and $ \nu \in \Nor(\Sigma(P_{a,x}),z) $ with $ \nu_{n+1}> 0 $, then elementary considerations imply that for every $ v \in \Real{n+1} $ with $ v \bullet \nu > 0 $ there are $ \epsilon > 0 $ and $ r > 0 $ such that
		\begin{equation*}
		\{ y : | s(y-z)-v|< \epsilon \; \textrm{for some $ s > 0 $}  \} \cap \mathbf{U}(z,r) \subseteq \Sigma^{+}(P_{a,x}),
		\end{equation*}
		whence we deduce that $ \Tan^{n}(\Haus{n}\restrict B,z) \subseteq \Real{n+1} \cap \{ w : w \bullet \nu \leq 0   \} $ by \cite[3.2.16]{MR0257325}. Since $ B $ is $ n $ rectifiable by \cite[4.12]{2017arXiv170309561M}, we deduce from \cite[3.2.19]{MR0257325} that for $ \Haus{n} $ a.e.\ $ z \in B $ there exists $ x \in C $ such that $ z \in \Sigma(P_{a,x}) $ and
		\begin{equation*}
		\Tan^{n}(\Haus{n}\restrict B,z) = \Tan(\Sigma(P_{a,x}),z).
		\end{equation*}
		Therefore for $ \Haus{n} $ a.e.\ $ z \in B $ there exists $ x \in C $ and an orthonormal basis $ e_{1}, \ldots , e_{n} $ of $ \Real{n} $ such that  
		\begin{equation*}
		e_{n} = \nabla P_{a,x}(z')/|\nabla P_{a,x}(z')| \quad \textrm{if $ \nabla P_{a,x}(z') \neq 0 $},
		\end{equation*}
		and
		\begin{equation*}
		\Big\{(e_{1}, 0),\; \ldots , \; (e_{n-1},0),\; \frac{(e_{n}, |\nabla P_{a,x}(z')|)}{(1 + |\nabla P_{a,x}(z')|^{2})^{1/2}}\Big\}
		\end{equation*}
		is an orthonormal basis of $ \Tan^{n}(\Haus{n}\restrict B,z) $. If $ p : B \rightarrow \Real{n} $ is given by $ p(z)= z' $ for $ z \in B $, noting that $ p $ is injective, the conclusion can be easily obtained applying \cite[3.2.20]{MR0257325} with $ W $ and $ f $ replaced by $ B $ and $ p $.
		\end{proof}
	
	\begin{proof*}[Proof of \ref{ABP}]
	Combine \ref{ABP arbitrary cod} and \ref{ABP : aux lemma}.
	\end{proof*}

\section{Proof of \ref{Extrinsic weak Harnack inequality}}\label{section: A measure estimate}

\begin{Lemma}\label{measure to point estimate}
	For every $ n \geq 1 $ there exist $ \alpha > 1 $ and $ 0 < \beta < 1 $ such that if $ 0 \leq h < a \leq \alpha^{-1} $, $ \Gamma $ is an $ (n,h) $ subset of $ \mathcal{C}_{1}(0) $ with $ \Haus{n}(\Gamma) < \infty $, $ x_{0} \in \mathbf{U}^{n}(0,1) $, $ r > 0 $ such that
	\begin{equation*}
	\mathbf{B}^{n}(x_{0},r) \subseteq \mathbf{U}^{n}(0,1) \quad \textrm{and} \quad A'_{a}(\Gamma) \cap \mathbf{U}^{n}(x_{0},r) \neq \varnothing,
	\end{equation*}
	then 
	\begin{equation*}
	\Leb{n}(A'_{\alpha a}(\Gamma) \cap \mathbf{U}^{n}(x_{0}, r/8)) \geq \beta r^{n}.
	\end{equation*}
\end{Lemma}
\begin{proof}
	Let $ x_{1} \in \mathbf{B}^{n}(0,1)	$ such that $ \Sigma(P_{a,x_{1}}) \cap \Gamma \cap \mathcal{C}_{r} (x_{0}) \neq \varnothing $. For every $ \gamma > 1 $ we let 
	\begin{equation*}  
	\phi(t) = 
	\begin{cases}
	-	\gamma^{-1}(16^{\gamma}-1) & \textrm{if $ 0 \leq t \leq 1/16 $}\\
	-\gamma^{-1}(t^{-\gamma} - 1) & \textrm{if $ 1/16 \leq t \leq 1 $}\\
	0 & \textrm{if $ t \geq 1 $} \\
	\end{cases}
	\end{equation*}
	and $ \psi(x) = P_{a,x_{1}}(x) + a r^{2} \phi(|x-x_{0}|/r) $ for every $ x \in \Real{n} $.
	
	Firstly, we prove the existence of a number $ 0 < \gamma < \infty $ depending only on $ n $ such that if $ h < a \leq (16^{\gamma + 1} + 2)^{-1} $ then\footnote{$Q_{\Sigma(\psi)}$ is the second fundamental of the $ \Sigma(\psi)$.} 
	\begin{equation}\label{measure to point estimate:1}
	\trace Q_{\Sigma(\psi)}(z) \bullet (-\nabla \psi(z'),1)  < -h(1 + |\nabla \psi(z')|^{2})^{1/2}
	\end{equation}
	for every $ z \in \Sigma(\psi) $ with $ r/16 < | z'-x_{0}| < r $. We fix $ x \in \Real{n} $ with $ r/16 < | x-x_{0}| < r $ and let $ z = (x, \psi(x)) $ and $  t = |x-x_{0}|/r $. We notice that
	\begin{equation*}
	|\nabla \psi(x)| \leq a(2+16^{\gamma+ 1}) \leq 1
	\end{equation*}
and if $ e_{1}, \ldots , e_{n} $ is an orthonormal basis of $ \Real{n} $ such that $ e_{n} = \nabla \psi(x)/ |\nabla \psi(x) | $ if $\nabla \psi(x) \neq 0 $, then
	\begin{equation*}
	\Big\{(e_{1}, 0),  \ldots ,  (e_{n-1}, 0), \frac{(e_{n}, |\nabla \psi(x)|)}{(1+|\nabla \psi(x)|^{2})^{1/2}} \Big\}
	\end{equation*}
	is an orthonormal basis of $ \Tan(\Sigma(\psi),z) $. Henceforth, letting $ \zeta = (1+|\nabla \psi(x)|^{2})^{-1/2}  $ and $ \nu = \zeta (-\nabla \psi(z'),1) $ and noting that $ 1/2 \leq \zeta \leq 1 $, we compute
	\begin{flalign*}
	& \trace Q_{\Sigma(\psi)}(z)\bullet \nu  = \zeta a (n-1) (1 + t^{-\gamma -2})  \\   & \qquad\qquad - \zeta a (\gamma + 2)t^{-\gamma - 2} \sum_{i=1}^{n-1}\bigg|\frac{x-x_{0}}{|x-x_{0}|} \bullet e_{i} \bigg|^{2} +  \zeta^{3}a (1 + t^{-\gamma -2}) \\ & \qquad\qquad - \zeta^{3}a (\gamma + 2)t^{-\gamma -2}\bigg|\frac{x-x_{0}}{|x-x_{0}|} \bullet e_{n} \bigg|^{2} \\
	&  \qquad \leq \zeta a(n-1)(1 + t^{-\gamma - 2}) -  a \zeta^{3}(\gamma + 2)t^{-\gamma - 2}  +  a \zeta^{3}(1 + t^{-\gamma - 2}) \\
	&  \qquad \leq an(1 + t^{-\gamma - 2}) -   \frac{a}{8}(\gamma + 2)t^{-\gamma - 2}.
	\end{flalign*} 
Since $ 0 \leq h < a $, we can evidently choose $ \gamma> 0 $ depending only on $ n $ so that the last line in the previous estimate is less than $ -h $ and \eqref{measure to point estimate:1} is proved.
	
	Next we prove that if $ h <a \leq (16^{\gamma + 1}+2)^{-1} $ then there exists $ z \in \Gamma \cap \overline{\mathcal{C}_{r/16}(x_{0})} $ such that 
	\begin{equation}\label{measure to point estimate:2}
	0 \leq P_{a,x_{1}}(z') - z_{n+1} \leq \gamma^{-1}(16^{\gamma} - 1) a r^{2}.
	\end{equation}
	In fact, since $ \Sigma^{+}(\psi) \cap \Gamma \cap \mathcal{C}_{r}(x_{0}) \neq \varnothing $ and $ \Sigma^{+}(\psi + ar^{2}\gamma^{-1}(16^{\gamma} - 1)) \cap \overline{\Gamma} = \varnothing $, we infer that there exists $ 0 < t \leq ar^{2}\gamma^{-1}(16^{\gamma} - 1) $ such that $ \psi + t $ touches $ \overline{\Gamma} $ from above. Since $ \psi(x) + t > P_{a,x_{1}}(x) $ if $ |x-x_{0}| \geq r $, we infer that
	\begin{equation*}
	\varnothing \neq \Sigma(\psi + t) \cap \Gamma \subseteq  \mathcal{C}_{r}(x_{0});
	\end{equation*}
	moreover, it follows from \ref{weak maximum principle} and \eqref{measure to point estimate:1} that
	\begin{equation*}
	\Sigma(\psi + t) \cap \Gamma \cap \big(\mathcal{C}_{r}(x_{0}) \sim \overline{\mathcal{C}_{r/16}(x_{0})}  \big) = \varnothing.
	\end{equation*}
	If $ z \in \Sigma(\psi + t) \cap \Gamma \cap \overline{\mathcal{C}_{r/16}(x_{0})} $ then
	\begin{equation*}
	z_{n+1}=\psi(z')+t=P_{a,x_{1}}(z') - ar^{2}\gamma^{-1}(16^{\gamma}-1) + t \leq P_{a,x_{1}}(z')
	\end{equation*}
	and \eqref{measure to point estimate:2} follows.

	We fix now $ a $, $ h $ and $ z $ as in \eqref{measure to point estimate:2} and let $ \theta > 0 $ to be chosen later depending only on $ \gamma $. For every $ y \in \mathbf{B}^{n}(z', r/64) $ we select $ c_{y} \in \Real{} $ such that the paraboloid
	\begin{equation*}
	Q_{y}(x) = P_{a,x_{1}}(x) +  \theta\frac{a}{2}|x-y|^{2}+c_{y}
	\end{equation*}
	touches $ \overline{\Gamma} $ from above. Noting that $ Q_{y}(z') \geq z_{n+1} $ and using \eqref{measure to point estimate:2}, we choose $ \theta $ large enough depending on $ \gamma $ so that $ Q_{y}(x) > P_{a,x_{1}}(x) $ whenever $ | x - z'| \geq r/16 $, which implies that
	\begin{equation*}
	\varnothing \neq \Sigma(Q_{y}) \cap \Gamma \subseteq \mathcal{C}_{r/16}(z') \subseteq \mathcal{C}_{r/8}(x_{0}).
	\end{equation*}
	Since $ Q_{y} $ is a paraboloid of center $ (\theta/(1+\theta))y +  (1/(1+\theta))x_{1} $ and opening $ (\theta+1)a $ we can apply \ref{ABP} with $ C $ replaced by 
	\begin{equation*}
	\bigg\{ \frac{\theta}{1 + \theta}y + \frac{1}{\theta + 1}x_{1} : y \in \mathbf{B}^{n}(z', r/64) \bigg\} \subseteq \mathbf{B}^{n}(0,1)
	\end{equation*}
	to obtain the conclusion.
\end{proof}

\begin{Remark}
	The proof of \ref{measure to point estimate} closely follows \cite[2.2]{MR2334822}, where continuous viscosity super-solutions of certain elliptic equations (including the minimal surface equation) are treated.
\end{Remark}

\begin{proof*}[Proof of \ref{Extrinsic weak Harnack inequality}]
This is a standard argument that combines the Vitali-type covering lemma in \cite[2.3]{MR2334822} with the measure-estimate \ref{measure to point estimate}. 

Let $ \alpha $ and $ \beta $ be as in \ref{measure to point estimate}, $ \epsilon = (48\alpha^{k+1})^{-1} $, $ \mu > 0 $ and let $ k $ be a positive integer to be chosen later depending only on $ n $ and $ \mu $. If $ P $ is the paraboloid of center $ 0 $ and opening $ 48\epsilon $ touching $ \overline{\Gamma} $ from above, we observe that the set of touching points of $ P $ with $\overline{\Gamma}$ is contained in $ \mathcal{C}_{1/3}(0) $. For every integer $ j \geq 0 $ we let $ F_{j} = \mathbf{B}^{n}(0,1/3) \cap A'_{48\epsilon\alpha^{j}}(\Gamma) $ and, noting that $ F_{0} \neq \varnothing $, we combine \cite[2.3]{MR2334822} (or \cite[6.4]{MR2757359}) and \ref{measure to point estimate} to conclude
\begin{equation*}
	\Leb{n}(\mathbf{B}^{n}(0,1/3) \sim A'_{\alpha^{-1}}(\Gamma)) \leq (1-\beta_{1})^{k}\Leb{n}(\mathbf{B}^{n}(0,1/3)),
\end{equation*}
where $ \beta_{1} $ depends only on $ \beta $. Now we choose $ k $ so that $ (1-\beta_{1})^{k} \leq \mu $.
\end{proof*}

\medskip 

\noindent Institut f\"ur Mathematik, Universit\"at Augsburg, \newline Universit\"atsstr.\ 14, 86159, Augsburg, Germany,
\newline mario.santilli@math.uni-augsburg.de

\end{document}